\documentclass[12pt]{amsart}

\usepackage{fullpage}
\usepackage{amsmath}
\usepackage{amsfonts}
\usepackage{amssymb}
\usepackage{amsthm}
\usepackage[all,knot,poly]{xy}
\usepackage{xypic}
\usepackage{hyperref}
\hypersetup{
    colorlinks=true,
    linkcolor=blue,
    filecolor=magenta,      
    urlcolor=cyan,
}
 
\input xy
\xyoption{all}

\theoremstyle{plain}

\newtheorem{theorem}{Theorem}[section]
\newtheorem{lemma}[theorem]{Lemma}
\newtheorem{proposition}[theorem]{Proposition}
\newtheorem{corollary}[theorem]{Corollary}
\newtheorem{remark}[theorem]{Remark}

\newtheorem{example}[theorem]{Example}
\newtheorem{conjecture}[theorem]{Conjecture}

\usepackage{color}

\title{{Commutator Equations in Finite Groups}}
\author{Kanto Irimoto}
\author{Enrique Torres-Giese}
\address{Trinity Western University, Langley BC, V2Y 1Y1 , Canada.}
\email{kanto.irimoto@twu.ca}
\email{enrique.torresgiese@twu.ca}

\begin{document}
\begin{abstract} 
The problem of finding the number of ordered commuting tuples of elements in a finite group
is equivalent to finding the size of the solution set of the system of equations determined 
by the commutator relations that impose commutativity among any pair of elements from
an ordered tuple. We consider this type of systems for the case of ordered triples and 
express the size of the solution set in terms of the irreducible characters of the group.  
The obtained formulas are natural extensions of Frobenius' character 
formula that calculates the number of ways a group element is a commutator of an 
ordered pair of elements in a finite group. We discuss how our formulas can be used to
study the probability distributions afforded by these systems of equations, and 
we show explicit calculations for dihedral groups.
\end{abstract}
\maketitle
\begin{center}
\today
\end{center}

\section{Introduction}
Systems of equations in finite groups have been studied using a variety of tools such as
probability and representation theory. One of the simplest systems of equations one can consider 
in a finite group $G$ 
is the one determined by the commutator equation: 
\[[x,y]=g, \]
where $g$ is a fixed element in $G$ and $[x,y]=x^{-1}y^{-1}xy$. 
According to Ore's Conjecture (proved in ~\cite{LOST}), this latter equation
always has a solution in finite non-abelian simple groups. If one considers the 
probability $P_2(g)$ that a randomly chosen ordered pair of elements in a finite group $G$ has commutator 
equal to $g$, then Ore's Conjecture is equivalent to saying that $P_2(g)$ is always positive for finite 
non-abelian simple groups. One of the key tools in proving Ore's Conjecture is
Frobenius's character formula which allows us to express the number of solutions to the equation $[x,y]=g$
in terms of the set $\mathrm{Irr}(G)$ of irreducible characters of $G$. More precisely, we have:
\begin{theorem}\label{frobenius}
[Frobenius] Suppose $g \in G$. Then
\[ |\{ (x,y) \in G\times G : [x,y]=g \} | =  \sum_{\chi \in \mathrm{Irr}(G)} \frac{|G|}{\chi(1)}\chi(g). 
\]
\end{theorem}
A complete proof of Frobenius' formula can be found in~\cite{dw}. 
If $k(G)$ is the number of conjugacy classes of $G$, then Frobenius' formula yields $P_2(1) = k(G)/|G|$, 
which is the probability of randomly selecting a commuting ordered pair.

This latter probability can be generalized in a number of ways, for instance we can consider the 
probability $P_n(1)$ that a randomly chosen $n$-tuple of elements in $G$ 
commutes:
\[ P_n(1) = \frac{ | \{ (x_1,\ldots,x_n) \in G^n : [x_i,x_j]=1 
\text{ for } i<j\} |}{ |G|^n}.\]
The probability $P_n(1)$ was also studied in ~\cite{lescot}
under the name of commutativity degree.
Note that the set of commuting $n$-tuples $\{ (x_1,\ldots,x_n) \in G^n : [x_i,x_j]=1 
\text{ for } i<j\}$ can be identified with the set of group homomorphisms
$\mathrm{Hom}(\mathbb{Z}^n,G)$.  Since $\mathrm{Hom}(\mathbb{Z}^{n+1},G)\subseteq 
\mathrm{Hom}(\mathbb{Z}^n,G)\times G\subseteq G^{n+1}$ it follows that
\[ 1\geq P_2(1) \geq \cdots \geq P_n(1) \geq P_{n+1}(1)\geq \cdots \]
with equality between any two (and hence between all) if and only if $G$ is abelian. 
Moreover, since $\mathrm{Hom}(\mathbb{Z}^n,G)\times 1 \subset \mathrm{Hom}(\mathbb{Z}^{n+1},G)$ we see that
$P_n(1)/|G| \leq P_{n+1}(1)$, and thus $P_2(1) \leq P_{2+n}(1)|G|^n$.
The sets $\mathrm{Hom}(\mathbb{Z}^n,G)$ reflect a number of structural properties of $G$
as well as topological properties. These sets have been used in ~\cite{acg} and ~\cite{act} to construct
the classifying space $B_{com}G$ for commutative $G$-bundles, and in ~\cite{etg} to study 
further properties of $P_n(1)$ related to $B_{com}G$.

To extend the definition of $P_n(1)$ to other elements $g\in G$, we will consider the class function  
\[ f_n(g)=| \{ (x_1,\ldots,x_n)\in G^{n} : [x_i,x_j] = g \text{ for } i< j \} |,\]
and we will set $P_n(g) = f_n(g)/|G|^n$. This latter is the probability that a randomly
chosen $n$-tuple in $G^n$ is a solution to the system of commutator equations $[x_i,x_j]=g$, for all
$i<j$.

The calculation of the value of $f_2(g)$ is given by Frobenius' formula, which is in terms
of the irreducible characters of $G$ and their degree. In contrast, calculating the value of $f_n(g)$ 
for different elements $g \in G$ when $n>2$ requires a careful analysis of the lattice of centralizers in 
$G$ and their cosets as we will 
show with the case $n=3$. Recall that if $f$ is a class function, then we can write it as 
$f = \sum_{\chi \in \mathrm{Irr}(G)} \alpha_\chi \chi$, where 
\[ \alpha_\chi = \langle f,\chi \rangle = \frac{1}{|G|} \sum_{g\in G}f(g)\overline{\chi(g)}.\]

\begin{theorem}\label{thm_one} Suppose $\chi \in \mathrm{Irr}(G)$. 
If we let
\[
\theta_{\chi}(a) = \sum_{b\in G} |C_G(ab)b\cap C_G(a)|\chi([a,b]),
\]
and
\[
m_\chi = \sum_{a\in G} \theta_\chi(a) ,
\]
then $\theta_\chi$ is a class function and 
\[ f_3(g) = \frac{1}{|G|}\sum_{\chi\in \mathrm{Irr}(G)} m_\chi \chi(g) .\]
\end{theorem}

The calculation of $f_3$ could be pretty involved, but it is possible to get an upper bound
for the value of $f_3(g)$ by considering the class function:
\[ t_3(g) = |\{(x,y,z)\in G^3: [x,y]=g=[x,z] \}| \]

Of course, $f_3(g)\leq t_3(g)$, and the calculation of $t_3$ in terms of the set $\mathrm{Irr}(G)$ is 
much simpler.

\begin{theorem}\label{thm_two} Suppose that $x_1,\ldots,x_{k(G)}$ is a full set of representatives of the
conjugacy classes of $G$. If $\chi \in \mathrm{Irr}(G)$, then
\[ \langle t_3, \chi \rangle = \sum_{i=1}^{k(G)} \frac{|G|}{\chi(1)}|\chi(x_i)|^2  
\]
\end{theorem}

As an application of $f_3$ we show that:

\begin{theorem}\label{thm_three}
If $n\geq 3$ and $g$ is an arbitrary element in the alternating group $A_n$, then the system of commutator equations:
\begin{align*}
[x_1,x_2] & = g \\
[x_1,x_3] & = g \\
[x_2,x_3] & = g 
\end{align*}
is always consistent in the symmetric group $\Sigma_n$.
\end{theorem}

\begin{remark}\normalfont
Based on a number of calculations in GAP~\cite{GAP4} we have conjectured that the system in Theorem~\ref{thm_three} is always consistent
in $A_n$ when $n\geq 5$. Further calculations also seem to indicate that this may be very well the case for any 
finite non-abelian simple group. This would be a natural extension of Ore's conjecture (see Remark~\ref{ore_extending}).  
\end{remark}

The organization of this paper is as follows: in Section 2 we prove the character formulas for $f_3$ and $t_3$; 
in Section 3 we calculate the coefficients of $f_3$ and $t_3$ for dihedral groups; and in Section 4 
we revisit some upper bounds for $P_n(1)$, obtain estimates for $P_3(g)$, and discuss some properties of the
probability distribution afforded by $f_3$.


\section{Character Formulas}

In order to write $t_3$ in terms of the irreducible characters of $G$ we will use the following 
formula (see Problem 3.12 of ~\cite{isaacs}):
\begin{equation}\label{formula}
\chi(g)\chi(h) = \frac{\chi(1)}{|G|}\sum_{z\in G}\chi(gh^z),
\end{equation}
where $\chi\in \mathrm{Irr}(G)$ and $h^z= z^{-1}hz$. Our first result in this section is a slight generalization of Theorem~\ref{thm_two}. 

\begin{theorem}\label{thm_twov2} 
Let $t_n(g) = | \{ (x_1,\ldots,x_n)\in G^n \colon [x_1,x_i]=g \text{ for } i>1 \}|$, and
suppose that $\vartheta$ is the character of the conjugation action of $G$ on itself.
If $\chi$ is an irreducible character of $G$, then
\[ \langle t_n, \chi \rangle = \frac{|G|}{\chi(1)} \langle \vartheta^{n-2}\cdot \chi,\chi\rangle    
\]
and hence $t_n$ is a character of $G$.
\end{theorem}
\begin{proof}
Note that $[x_1,x_2]=[x_1,x_i]$ if and only if $x_2 x_i^{-1} \in C_G(x_1)$ for all $i\geq 2$. This implies that
the set that affords the value of $t_n(g)$ is determined by $g$ and pairs $(x_1,x_2)$ with $[x_1,x_2]=g$. 
Then using (\ref{formula}) we have
\begin{align*}
\langle t_n, \overline{\chi} \rangle = \frac{1}{|G|}\sum_{g\in G} t_n(g)\chi(g) & = \frac{1}{|G|} \sum_{x_1\in G}\sum_{x_2\in G} |C_G(x_1)|^{n-2} \chi([x_1,x_2]) \\
 & = \frac{1}{|G|} \sum_{x_1\in G}|C_G(x_1)|^{n-2}\frac{|G|}{\chi(1)}\chi(x_1^{-1})\chi(x_1) \\
 & = \frac{|G|}{\chi(1)}\left( \frac{1}{|G|} \sum_{x_1\in G}(\vartheta^{n-2}(x_1)\chi(x_1))\overline{\chi(x_1)} \right).
\end{align*}
So $\langle t_n, \overline{\chi} \rangle$ is $|G|/\chi(1)$ times the multiplicity of $\chi$ in $\vartheta^{n-2}\cdot\chi$, which is
the product of two non-negative integers. 
Thus $\langle t_n, \overline{\chi} \rangle = \overline{\langle t_n, \overline{\chi} \rangle }= 
\langle t_n, \chi \rangle$. This completes the proof.
\end{proof}

\begin{remark}\normalfont
If we fix $g \in G$, then there is a bijection between the sets 
\[\{ (x,y,z)\in G^3 : [x,y]=[x,z]=[y,z]=g \} \] 
and
\[ \{ (x,y,z)\in G^3 : [x,y]=[x,z]=[z,y]=g \} \] 
given $(x,y,z) \mapsto (x,z,y)$. 
\end{remark}

Now we proceed to prove our formula for $f_3$. \\

\noindent {\bf Proof of Theorem~\ref{thm_one}:}
Let $T_g$ be the set $\{ (x,y,z)\in G^3 : [x,y]=[x,z]=[z,y]=g \}$. Then according to the previous remark $f_3(g) = |T_g|$.
Suppose that $(x,y,z)\in \sqcup_{g\in G} T_g$. Then, the condition $[x,y]=[x,z]$ holds if and only if $y=cz$ for some $c\in C_G(x)$, or equivalently 
$y=cz$ for some $c$ such that $x\in C_G(c)$. 
In addition, $[x,y]=[z,y]$ if and only if $[y,x]=[y,z]$, if and only if
$x\in C_G(y)z$. Then in order to form a triple $(x,y,z)$ in a set $T_g$ we can first pick any two elements
$z,c\in G$, then set $y=cz$ and find $x\in G$ such that $x\in C_G(cz)z \cap C_G(c)$. Then we have
\[
\sum_{g\in G} f_3(g)\chi(g) = \sum_{c\in G}\sum_{z\in G} |C_G(cz)z\cap C_G(c)|\chi([z,cz]). 
\]
Since $[z,cz] = [z,z][z,c]^z$ and $\chi( [z,c]) = \overline{\chi( [c,z] )}$, it follows that we can write this latter as:
\[
|G|\langle f_3,\chi \rangle = \sum_{g\in G} f_3(g)\overline{\chi(g)} = \sum_{a\in G}\sum_{b\in G} |C_G(ab)b\cap C_G(a)|\chi([a,b]) 
\]
Hence
\[
\langle f_3,\chi\rangle = \frac{1}{|G|} \sum_{a \in G}\theta_\chi(a) = \frac{m_\chi}{|G|}, 
\]
as wanted. To show that $\theta_\chi$ is a class function fix an element $w$ in $G$ and note that:
\begin{align*} 
\theta(a^w) & = \sum_{b\in G} |C_G(a^w b)b\cap C_G(a^w)|\chi([a^w,b]) \\
 & = \sum_{b^w\in G} |C_G(a^w b^w)b^w\cap C_G(a^w)|\chi([a^w,b^w]) \\
 & = \theta(a)
\end{align*}
\hfill $\square$

When studying systems of commutator equations in a group $G$ there will be times when we will consider 
solutions consisting of tuples of group elements in specific subgroups of $G$. This motivates the
following class function: let $H$ be a subgroup of $G$ and let $f_{n,H\leq G}\colon G\to\mathbb{N}$ given by
\[
f_{n,H\leq G}(g)= | \{ (x_1,\ldots,x_n)\in H^{n} : [x_i,x_j] = g \text{ for } i< j \} |.
\]
Note that $f_{n,G\leq G} = f_n$. When $H=G$ we will write $f_{n,G}$ or simply $f_n$.
Now we can prove multiple properties of the functions $f_n$ and $t_n$.

\begin{proposition}\label{properties} 
Suppose that $g\in G$, and that $H$ and $K$ are subgroups of $G$. We have the following:
\begin{enumerate}
\item If $H\leq K$, then $f_{n,H\leq G}(g)\leq f_{n,K\leq G}(g)$.
\item If $\chi \in \mathrm{Irr}(G)$ and we let
\[
\tau_\chi(b) = \sum_{a\in G} |C_G(ab)b\cap C_G(a)|\chi([a,b]),
\]
then $\tau_\chi$ is a class function and $\tau_\chi(b^{-1}) = \tau_{\overline{\chi}}(b) = \overline{\tau_\chi(b)}$.
\item $m_\chi = \sum_{b\in G} \tau_\chi(b)$ and $m_\chi \in \mathbb{R}$.
\item Both $f_n$ and $t_n$ are invariant under isoclinism between groups of the same order. In particular, $P_n$ is invariant
under isoclinism. 
\item $t_n(g) \leq t_n(1)$.
\item $f_3(g) \leq t_3(g)$, and if $g\neq 1$ then $f_3(g) \leq t_3(g) - f_2(g)$.
\item $f_n(g^{-1}) = f_n(g)$. 
\item $t_n(g^{-1}) = t_n(g)$.
\end{enumerate}
\end{proposition}
\begin{proof} 
\begin{enumerate}
\item This statement is straightforward. 
\item To prove that $\tau_\chi$ is a class function 
we can proceed as we did for $\theta_\chi$ in Theorem~\ref{thm_one}, 
so we leave this to the interested reader. 
Note that for any two non-empty subsets $A,B$ of $G$ and any element $v$ in $G$ we
have $|Av \cap B| = |A\cap Bv^{-1}|$. Then
\begin{align*}
\tau_\chi(b^{-1}) & = \sum_{a\in G} |C_G(ab^{-1})b^{-1}\cap C_G(a)|\chi([a,b^{-1}]) \\
 & = \sum_{a\in G} |C_G(ab^{-1}) \cap C_G(a)b|\chi([a,b^{-1}]) \\
 & =\sum_{a^{-1}\in G} |C_G(a^{-1}b^{-1})\cap C_G(a^{-1})b|\chi([a^{-1},b^{-1}]) \\
\end{align*}
So if we set $u = a^{-1}b^{-1}$ and use the identities $[ub,b^{-1}] = [u,b^{-1}]^b$ and
$[u,b^{-1}] = [b,u]^{b^{-1}}$, we get:
\begin{align*}
\tau_\chi(b^{-1}) & =\sum_{u\in G} |C_G(u) \cap C_G(ub)b|\chi([ub,b^{-1}]) \\
 & =\sum_{u\in G} |C_G(u)\cap C_G(ub)b|\chi([u,b^{-1}]) \\
 & =\sum_{u\in G} |C_G(u)\cap C_G(ub)b|\chi([b,u]) \\
 & =\sum_{u\in G} |C_G(u)\cap C_G(ub)b|\overline{\chi([u,b])} \\
 & = \overline{\tau_\chi(b)} = \tau_{\overline{\chi}}(b).
\end{align*}
\item That $m_\chi = \sum_{b\in G} \tau_\chi(b)$ follows from the definition of $m_\chi$. 
Moreover, we have:
\[
\overline{m_\chi} = \sum_{b\in G}\overline{\tau_\chi(b)} = \sum_{b^{-1}\in G}\overline{\tau_\chi(b^{-1})} =
\sum_{b^{-1}\in G}\tau_\chi(b)= m_\chi.
\]
\item To prove that $f_n$ is invariant under isoclinism between groups of the same order we refer the reader to the 
proof in~\cite{lescot} of the invariance of the value of $P_n(1)$. The argument for $t_n$ is the same.

\item If we write $t_n = \sum_{\chi \in \mathrm{Irr}(G)} a_\chi \chi$, then according to Theorem~\ref{thm_twov2} 
the coefficients $a_\chi$ 
are non-negative integers. Thus, 
\[ t_n(g) \leq \sum_{\chi \in \mathrm{Irr}(G)} a_\chi \chi(1) = t_n(1) ,\]
as wanted. 

\item The first inequality is straightforward, while for the second note that if $g\neq 1$ then a triple $(x,y,z)$ 
satisfying $[x,y]=[x,z]=[y,z]=g$ cannot have $y=z$. The triples that do satisfy $y=z$ are of the form $(x,y,y)$ and can be
counted by the function $f_2(g)$. These latter triples can be knocked off from the set of triples that are counted by $t_3$ yielding
a set containing the set of triples that are counted by $f_3$.

\item The map given by $(a_1,\ldots,a_n)\mapsto (a_n,\ldots,a_1)$ defines a bijection between the set $\{(x_1,\ldots,x_n)\colon [x_i,x_j]=g \text{ for }i<j\}$ and the set
$\{ (x_1,\ldots,x_n): [x_i,x_j] = g^{-1} \text{ for }i<j \}$. Their size is precisely $f_n(g)$ and $f_n(g^{-1})$. 
\item If we write $t_n = \sum_{\chi \in \mathrm{Irr}(G)} a_\chi \chi$, then
\[ t_n(g^{-1}) = \sum_{\chi \in \mathrm{Irr}(G)} a_\chi \overline{\chi(g)} = \overline{\sum_{\chi \in \mathrm{Irr}(G)} a_\chi \chi(g)} = 
\overline{t_n(g)} = t_n(g). \]
\end{enumerate}
\end{proof}

\begin{example}\label{coeff_a5}\normalfont
Using Proposition~\ref{properties} one can simplify the calculations yielding the coefficients of $f_3$ and $t_3$. 
For the alternating group $A_5$ we have:
\begin{enumerate}
\item $f_2 = 60\chi_1 + 20\chi_2 + 20\chi_3 + 15\chi_4 + 12\chi_5$.
\item $f_3 = 40\chi_1 + 64\chi_2 + 64\chi_3 + 84\chi_4 + 112\chi_5$.
\item $t_3 = 300\chi_1 + 260\chi_2 + 260\chi_3 + 285\chi_4 + 324\chi_5$. 
\item $P_2(1)=1/12$ and $P_3(1) = 11/1800$.
\end{enumerate}
Here $\chi_1$ is the character of the trivial representation, $\chi_2$ and $\chi_3$ have degree 3, 
whereas $\chi_4$ and $\chi_5$ have 
degree 4 and 5 respectively.
\end{example}

\begin{remark}\normalfont
We have done numerous calculations in GAP and we have conjectured that $f_3$ is always a character.
\end{remark}


\section{Calculations for Dihedral groups}

In this section we calculate the coefficients of $f_3$ and $t_3$ for the dihedral group: 
$D_{2n} = \langle a,b\colon a^n=b^2=1, a^b = a^{-1} \rangle.$
For convenience of the reader we include the character table(s) of $D_{2n}$. For more details we refer the
reader to~\cite{liebeck}.

When $n$ is odd, $D_{2n}$ has two linear characters $\chi_1,\chi_2$, and $(n-1)/2$ degree-two irreducible
characters $\psi_1,\ldots,\psi_{(n-1)/2}$. The conjugacy classes in this case are: $\{ 1\}$; $\{ a^r,a^{-r} \}$ for 
$1\leq r\leq (n-1)/2$; and $\{ a^s b : 0 \leq s\leq n-1 \}$. If we set $\omega = e^{2\pi i/n}$, we have the 
following table.

$$
\begin{array}{c|ccr}
  \rm class& 1 & a^r & b \cr
  \rm size&1&2& n\cr
\hline
  \chi_{1}&1&1&1 \cr
  \chi_{2}&1&1&-1\cr
  \psi_{j}&2& \omega^{jr} + \omega^{-jr}&0\cr
\end{array}
$$

When $n$ is even, $D_{2n}$ has four linear characters $\chi_1,\ldots,\chi_4$ and
$(n-2)/2$ degree two irreducible characters $\psi_1,\ldots,\psi_{(n-2)/2}$. In this case, if we write 
$n=2l$, then the conjucagy classes are as follows: $\{ 1\}$; $\{ a^l\}$; $\{ a^r, a^{-r}\}$ for $1\leq r\leq l-1$; 
$\{ a^r b\colon r \text{ even}\}$; and $\{ a^s b: s \text{ odd}\}$. We have the following table.

$$
\begin{array}{c|cccrr}
  \rm class& 1 & a^l & a^r & b & ab \cr
  \rm size&1&1& 2& l &  l\cr
\hline
  \chi_{1}&1&1&1 & 1 & 1 \cr
  \chi_{2}&1&1&1&-1&-1\cr
  \chi_{3}&1&(-1)^l&(-1)^r&1&-1\cr
  \chi_{4}&1&(-1)^l&(-1)^r&-1&1\cr
  \psi_{j}&2&2(-1)^j &  \omega^{jr} + \omega^{-jr}&0&0\cr
\end{array}
$$



\begin{lemma}\label{trig}
If $n$ is a positive integer, then
\[ \sum_{k=1}^{n-1} \cos \left( \frac{2k\pi}{n}\right) = -1; \]
and if $n$ is an odd integer greater than 1, then we also have 
\[ \sum_{k=1}^{\frac{n-1}{2}} 2 \cos\left( \frac{4k\pi}{n}\right) = -1. \]
\end{lemma}
\begin{proof}
Recall that $\omega = e^{2\pi i/n} = \cos{\frac{2\pi}{n}} + i\sin{\frac{2\pi}{n}}$. Then, $\omega^n - 1 = (\omega - 1)\sum_{k=0}^{n-1} \omega^k = 0$.
So it follows that
\[ \sum_{k=1}^{n-1} \cos\left( {\frac{2k\pi}{n}}\right) = -1.\]
If $n$ is a positive odd integer greater than 1, then 
\[
\sum_{k=1}^{\frac{n-1}{2}} 2 \cos \left(\frac{4k\pi}{n}\right) =  \sum_{k=1}^{n-1} \cos\left(\frac{2k\pi}{n}\right) = -1.
\]
\end{proof}

\begin{theorem}\label{coeffs_dihedral}
The coefficients of $f_3$ for $D_{2n}$ are given as follows: 
\begin{enumerate}
\item When $n \equiv 1,3 \mod 4$,
\[ \langle f_3,\chi_i \rangle = \frac{1}{2}(n^2+2n+5) \text{ and } \langle f_3,\psi_j \rangle = n^2+5. \]
\item When $n \equiv 0 \mod 4$,
\[ \langle f_3,\chi_i \rangle = \frac{1}{2}(n^2+4n+24) \text{ and } \langle f_3,\psi_j \rangle = \left\{ 
\begin{array}{l}
n^2+16 \text{ if $j$ is odd}\\
n^2+24 \text{ if $j$ is even.}
\end{array}
\right. \]
\item When $n \equiv 2 \mod 4$,
\[ \langle f_3,\chi_i \rangle = \frac{1}{2}(n^2+4n+20) \text{ and } \langle f_3,\psi_j \rangle = n^2+20. \]
\end{enumerate}
\end{theorem}

\begin{proof}
Let $\chi$ be an irreducible character of $D_{2n}$. To simplify notation we will write only $C(y)$ to denote 
the centralizer of $y$ in $D_{2n}$. 
According to Theorem~\ref{thm_one} we have:
\[ \langle f_3,\chi \rangle = \frac{1}{|D_{2n}|}\sum_{x\in D_{2n}} \theta_\chi(x), \text{ where } 
\theta_\chi(x) = \sum_{y\in D_{2n}} |C(xy)y \cap C(x)|\chi([x,y]).  \]
We will make use of Lemma~\ref{trig}, and of the fact that $\theta_\chi$ is a class function
(see Theorem~\ref{thm_one}). 

Suppose $n \equiv 1,3 \mod 4$. A careful analysis shows that we have the following calculations:
\begin{align*} 
\theta_\chi(1) &=\sum_{y \in G} |C(y)|\chi(1)\\
&= [2n+(n-1)n+2n]\chi(1)\\
&=(n^2+3n)\chi(1).
\end{align*}
\begin{align*}
\theta_\chi(a^r)&=\sum_{y \in G} |C(a^ry)y \cap C(a^r)|\chi([a^r,y])\\
&=n^2\chi(1)+n\chi(a^{2r}).
\end{align*}
\begin{align*}
\theta_\chi(a^rb)&=\sum_{y \in G} |C(a^rby)y \cap C(a^rb)|\chi([a^rb,y])\\
&= 4\chi(1) + \sum_{i=1}^{n-1} \chi(a^i).
\end{align*}
Thus,
\begin{align*}
\langle f_3,\chi \rangle &= \frac{1}{2n} \left[ (n^2+3n)\chi(1) + \sum_{i=1}^{n-1} [n^2\chi(1) + n \chi(a^{2i})] + n[4\chi(1) + \sum_{i=1}^{n-1} \chi(a^i)] \right]\\
&= \frac{1}{2n} \left[ (n^3+7n)\chi(1) + \sum_{i=1}^{\frac{n-1}{2}} 2n \chi(a^{2i}) + n\sum_{i=1}^{n-1} \chi(a^i) \right] \\
&= \frac{1}{2} \left[ (n^2+7)\chi(1) + \sum_{i=1}^{\frac{n-1}{2}} 2 \chi(a^{2i}) + \sum_{i=1}^{n-1} \chi(a^i) \right].
\end{align*}

Hence, if $\chi= \chi_i$, then
\begin{align*}
\langle f_3,\chi_i \rangle &= \frac{1}{2} \left[ (n^2+7) + (n-1) + (n-1) \right]\\
&= \frac{1}{2}(n^2+2n+5),
\end{align*}
and if $\chi=\psi_i$, then
\begin{align*}
\langle f_3,\psi_i \rangle &= \frac{1}{2} \left[ 2(n^2+7) +\sum_{i=1}^{\frac{n-1}{2}} 4 \cos\left( \frac{4i\pi}{n}\right) 
+ \sum_{i=1}^{n-1} 2\cos\left( \frac{2i\pi}{n}\right) \right]\\
&=n^2+7-1-1\\
&= n^2 + 5.
\end{align*}

Likewise, when $n \equiv 0 \mod 4$, we have the following calculations:
\begin{align*}
\theta_\chi(1)&=\sum_{y \in G} |C(y)|\chi(1)\\
&= [2\cdot2n+(n-2)n+2n]\chi(1)\\
&=(n^2+6n)\chi(1).
\end{align*}
\begin{align*}
\theta_\chi(a^{\frac{n}{2}})&=\sum_{y \in G} |C(y)|\chi(1)\\
&=(n^2+6n)\chi(1).
\end{align*}
\begin{align*}
\text{When }r\neq \frac{n}{2}: \hspace{0.5cm} \theta_\chi(a^r)&=\sum_{y \in G} |C(a^ry)y \cap C(a^r)|\chi([a^r,y])\\
&=n^2\chi(1)+2n\chi(a^{-2r}).
\end{align*}
\begin{align*}
\text{When $r$ is even:}\hspace{0.5cm}\theta(a^rb) &=\sum_{y \in G} |C(a^rby)y \cap C(a^rb)|\chi([a^rb,y])\\
&=2\cdot4\chi(1) + 2\cdot4\chi(1) + 4\chi(a^{\frac{n}{2}}) + \sum_{i=1}^{\frac{n}{2}-1}4\chi(a^{2i})\\
&=16\chi(e) + 4\chi(a^{\frac{n}{2}}) + \sum_{i=1}^{\frac{n}{2}-1}4\chi(a^{2i}).
\end{align*}
\begin{align*}
\text{When $r$ is odd:}\hspace{0.5cm} \theta_\chi(a^rb) &=16\chi(1) + 4\chi(a^{\frac{n}{2}}) 
+ \sum_{i=1}^{\frac{n}{2}-1}4\chi(a^{2i}).
\end{align*}

Thus, 
\begin{align*}
\langle f_3,\chi \rangle &= \frac{1}{2n} \left[ 2(n^2+6n)\chi(1) + (n-2)n^2\chi(e) + \sum_{i=1}^{\frac{n}{2}-1} 4n \chi(a^{2i}) \right. \\
&+ \left. 2\cdot \frac{n}{2}(16\chi(1) + 4\chi(a^{\frac{n}{2}}) + \sum_{i=1}^{\frac{n}{2}-1}4\chi(a^{2i})) \right] \\
&= \frac{1}{2} \left[ (n^2+28)\chi(1) + 4\chi(a^{\frac{n}{2}}) +  \sum_{i=1}^{\frac{n}{2}-1}8\chi(a^{2i}) \right].
\end{align*}

Hence, when $\chi=\chi_i$, we have
\begin{align*}
\langle f_3,\chi_i \rangle &= \frac{1}{2} \left[ (n^2+28) + 4 +  \sum_{i=1}^{\frac{n}{2}-1}8 \right]\\
&= \frac{1}{2}(n^2+4n+24),
\end{align*}
and when $\chi=\psi_j$ we get
\begin{align*}
\langle f_3,\psi_j \rangle &= \frac{1}{2} \left[ 2(n^2+28) + 8\cos(j\pi) + 
\sum_{i=1}^{\frac{n}{2}-1} 16 \cos\left( \frac{2i\pi j}{n}\right) \right]\\
&= \left\{
\begin{array}{l}
n^2+16 \text{ if $j$ is odd}\\
n^2+24 \text{ if $j$ is even.}
\end{array}
\right.
\end{align*}

Finally, when $n \equiv 2 \mod 4$, a similar analysis yields:
\begin{align*}
\langle f_3,\chi \rangle &= \frac{1}{2n} \left[ 2(n^2+6n)\chi(e) + (n-2)n^2\chi(1) 
+ 2\cdot \frac{n}{2}(16\chi(1) + 4\chi(a^{\frac{n}{2}}) + \sum_{i=1}^{\frac{n}{2}-1}4\chi(a^{2i})) \right] \\
&= \frac{1}{2} \left[ (n^2+28)\chi(1) + \sum_{i=1}^{\frac{n}{2}-1}8\chi(a^{2i}) \right].
\end{align*}
Thus, if $\chi=\chi_i$, then
\[
\langle f_3,\chi \rangle = \frac{1}{2}(n^2+4n+20).
\]
And if $\chi=\psi_j$, then
\[
\langle f_3,\chi \rangle = n^2+20.
\]
\end{proof}

\begin{theorem}\label{f3rotations} 
Suppose that $n\geq 3$. If $g$ is in the commutator subgroup of $D_{2n}$, then $f_3(g)>0$.
More precisely, we have the following:
\begin{enumerate}
\item \[
f_3(1) = \left\{
\begin{array}{c c}
n^3+7n & \text{ if $n$ is odd,}\\
n^3 + 28n & \text{ if $n$ is even.}
\end{array}
\right.
\]
\item
For any integer $s$ with $0 < s < \frac{n}{2}$,
\[
f_3(a^{2s}) =
\left\{
\begin{array}{cc}
2n &\text{ if $n$ is odd,}\\
8n &\text{ if $n$ is even and $s \neq \frac{n}{4}$,}\\
12n &\text{ if $n$ is even and $s = \frac{n}{4}$.}
\end{array}
\right.
\]
\end{enumerate}

\end{theorem}
\begin{proof}
Recall that $[D_{2n},D_{2n}] = \langle a^2 \rangle$. 
We will go over the case when $n$ is odd and will leave 
the cases $n \equiv 0,2 \mod 4$ to the reader. 

When $n$ is odd it suffices to calculate $f_3(a^{2s})$ where $0\leq 2s\leq (n-1)/2$. 
According to Theorem~\ref{coeffs_dihedral}, when $n\equiv 1,3 \mod 4$ we have:
\[ f_3(a^{2s}) = \frac{1}{2}(n^2 + 2n + 5)(1+1) + (n^2+5)\sum_{k=1}^{\frac{n-1}{2}} 2\cos\left( \frac{4\pi s k}{n}  \right) 
= 2n,
\]
and 
\[ f_3(1) = \frac{1}{2}(n^2 + 2n + 5)(1+1) + (n^2+5)\left( \frac{n-1}{2}\right) 2 = n^3 + 7.
\]
\end{proof}

Recall that the coefficients of $t_3$ are given by 
\[ \langle t_3, \chi \rangle = |G| \sum_{x \in K} \frac{|\chi(x)|^2}{\chi(1)}, \]
where $K$ is a system of representatives of the conjugacy classes of $G$, and $\chi \in \mathrm{Irr}(G)$. Using this formula we easily obtain the following result, whose proof 
we omit.

\begin{theorem}
The coefficients of $t_3$ for the dihedral group $D_{2n}$ are given as follows: 
\begin{enumerate}
\item When $n$ is odd,
\[ \langle t_3,\chi_i \rangle = n^2+3n,\text{ and  } \langle t_3,\psi_i \rangle = n^2+2n,\]
\item and when $n$ is even,
\[ \langle t_3,\chi_i \rangle = n^2+6n,\text{ and  }  \langle t_3,\psi_i \rangle = n^2+4n.\]
\end{enumerate}
\end{theorem}


\section{Probability Distributions} 

Calculating the exact value of $f_3(g)$, or equivalently that of $P_3(g)$, could be a pretty challenging 
task even in the case when $g=1$.
Nevertheless, it is possible to obtain some estimates as we will show in this section. 
Some results estimating $P_2(1)$ and $P_2(g)$ can be found for instance 
in~\cite{gurnalick_robinson} and~\cite{ps}.

A key observation that has been used to estimate $P_n(1)$
is to write $f_n(1)$ recursively as follows:
\[ f_n(1) = 
|\{ (x_1,\ldots,x_n)\in (G\setminus Z(G))\times G^{n-1}\colon [x_i,x_j]=1 \text{ for } i<j \} |
+
|Z(G)|f_{n-1}(1).
\]
A similar recursive formula can be obtained if we consider the function $f_n$ restricted to
tuples formed with elements in the centralizers of elements of the group $G$. More precisely,
\[
f_n(1) = \sum_{g \in G} f_{n-1,C_G(g)}(1).  
\]

\begin{example}\normalfont
Recall that a group $G$ is called TC (for transitively commutative) if 
commutativity is a transitive relation on the set of non-central elements
of $G$. This latter condition is equivalent to requiring non-central elements
to have abelian centralizers. For TC groups (also known as CA-groups) both recursive formulas 
simplify to:
\[ f_n(1) = |G| \sum_{x_i\notin Z(G)} |C_{G}(x_i)|^{n-2} + |Z(G)|f_{n-1}(1). 
\]
where $x_1,\ldots,x_k$ is a full set of representatives of the conjugacy classes of 
$G$. 
\end{example}

It is also possible to approach the calculation of $P_n(1)$ by considering the 
poset of abelian subgroups of $G$ as was shown in~\cite{etg}. The structure of this poset 
simplifies when $G$ is a TC group. For instance, we showed in~\cite{etg} that for the
alternating group $A_5$ (which is a TC group) we have:
\[
P_n(1)= \frac{6}{12^n} + \frac{5}{15^n} + \frac{10}{20^n} - \frac{20}{60^n}.
\]

It is well-known that $P_2(1)\leq 5/8$ for any non-abelian group $G$ (see for instance~\cite{gustafson}), 
and that this upper bound is attained by groups that satisfy $G/Z(G)\cong \mathbb{Z}_2\times \mathbb{Z}_2$.
This upper bound was extended in~\cite{lescot} to: 
\[
P_n(1) \leq \frac{3\cdot 2^{n-1}-1}{2^{2n-1}}.\]
This inequality can be slightly improved if we take into account the index of the center of $G$.
We will show this in the following result only when $n=3$, although it is not hard to extend it to higher values 
of $n$.

\begin{proposition}
Suppose that $G$ is a finite non-abelian group. Then for all $g\in G$, we have:
\begin{enumerate}
\item $P_3(1) \leq \frac{1}{2}(P_2(1) - \alpha) + \alpha P_2(1) \leq \frac{11}{32}$, where $\alpha^{-1} = |G:Z(G)|$.
\item $\frac{1}{|G||G'|} \leq P_3(g) \leq \frac{P_2(1)}{|G|}\sum_{\chi \in \mathrm{Irr}(G)} \chi(1)|\chi(g)| \leq P_2(1)$.
\item $P_3(g) \leq P_2(1)\sqrt{\frac{|C_G(g)|}{|G|}}$.
\end{enumerate}
\end{proposition}
\begin{proof}
\begin{enumerate}
\item 
If we write the class equation of $G$ as $|G| = |Z(G)| + s_1 + \cdots + s_m$, then each $s_i$ is at least two
and so $m \leq (|G| - |Z(G)|)/2$. Since $G$ is not abelian it follows that $\alpha \leq 1/4$, and hence
\[ P_2(1) = \frac{k(G)}{|G|} = \frac{|Z(G)| + m }{|G|} \leq \frac{1+\alpha}{2}. \]
Now we apply this idea again:
\begin{align*} 
P_3(1) & = \sum_{ i=i }^{m} \frac{|G|}{|C_G(x_i)|} \frac{f_{2,C_G(x_i)}(1)}{|G|^3} + \frac{|Z(G)|}{|G|}P_2(1) \\
 & = \sum_{ i=i }^{m} \frac{|G||C_G(x_i)|}{|G|^3} \frac{f_{2,C_G(x_i)}(1)}{|C_G(x_i)|^2} + \alpha P_2(1) \\
 & \leq (k(G)- |Z(G)|)\frac{1}{|G|}\cdot \frac{1}{2}\cdot 1 + \alpha P_2(1) \\
 & = \frac{1}{2}(P_2(1) - \alpha) + \alpha P_2(1) \leq \frac{1}{2}\left( \frac{1 -\alpha}{2} \right) + 
 \alpha \left( \frac{1+\alpha}{2}  \right) \\
 & = \frac{2\alpha^2 + \alpha + 1 }{4} \leq \frac{11}{32}.  
\end{align*}
\item Note that
\[
\frac{|\theta_\chi(x)|}{|C_G(x)|} \leq \sum_{y \in G} \chi(1) = |G|\chi(1).
\]
Thus
\[
| m_\chi | \leq \sum_{x \in G} |\theta_\chi(x) | = 
 \sum_{x^G} \frac{|G|}{|C_G(x)|} |\theta_\chi(x)|  \leq \sum_{x^G} |G||G|\chi(1). 
\]
and so 
\[ 
 \frac{| m_\chi |}{|G|^3} \leq P_2(1)\chi(1).
\]
It follows that 
\[ P_3(g) \leq \frac{1}{|G|}\sum_{\chi \in \mathrm{Irr}(G)} \frac{|m_\chi|}{|G|^3}\chi(g) \leq
\frac{1}{|G|}P_2(1)\sum_{\chi \in \mathrm{Irr}(G)} \chi(1)|\chi(g)| \leq P_2(1).\] 
The lower bound follows from $P_2(1)/|G| \leq P_3(1)$ and $1/|G'| \leq P_2(1)$ according to~\cite{gurnalick_robinson}.
\item Suppose that $K$ is a full system of representatives of the conjugacy classes of $G$. Then  by Theorem~\ref{thm_two} 
and applying the Cauchy-Schwarz inequality we have:
\[ \langle t_3, \chi \rangle = |G| \sum_{x \in K} \frac{|\chi(x)|^2}{\chi(1)} \leq |G| \sum_{x \in K} \chi(1)
\leq |G| k(G)^{1/2} |G|^{1/2}. \]
Thus, 
\begin{align*}
t_3(g) &= \sum_{\chi \in \mathrm{Irr}(G) } \langle t_3, \chi \rangle \chi(g) \\
& \leq |G|^{3/2}k(G)^{1/2} \sum_{\chi \in \mathrm{Irr}(G)} |\chi(g)| \\
& \leq |G|^{3/2}k(G)^{1/2}k(G)^{1/2} |C_G(g)|^{1/2}, 
\end{align*}
where the last inequality is an application of the Cauchy-Schwarz inequality.
The desired inequality follows from the inequality $f_3(g) \leq t_3(g)$ (see Proposition~\ref{properties}).  
\end{enumerate}
\end{proof}

Unlike the values of $P_2(g)$, the set of probability values $\{ P_3(g) \colon g\in G\}$ does not constitute a probability distribution on $G$.
Nevertheless, we can normalize $f_3$ to define a probability distribution on a group $G$ by setting:
\[
Q_3(g) = \frac{f_3(g)}{\sum_{x \in G} f_3(x)}
\]
Using the coefficients of $f_3$ we can also write $Q_3$ in terms of $\mathrm{Irr}(G)$. For instance,
for the alternating group $A_5$ we have (see Example~\ref{coeff_a5}):
\[ Q_3(g) =
\frac{1}{60}\left( \chi_1 + \frac{8}{5}\chi_2 + \frac{8}{5}\chi_3 + \frac{21}{10}\chi_4 + 
\frac{14}{5}\chi_5  \right). \]

It has been shown in ~\cite{shalev} that the distribution $P_2(g)= \frac{f_2(g)}{|G|^2}$ converges in the $L_1$-norm
to the uniform distribution $U(g)=\frac{1}{|G|}$ for finite non-abelian simple groups as $|G|\to \infty$.
Several computations seem to indicate that this latter is not the case for the 
distribution $Q_3$. For instance, we can see in the chart below that the distribution of $Q_3$ on
$A_5$ is heavily skewed at the identity (all the percentages are approximations).

\[
\begin{array}{c|rrrrr}
  \rm class&1&(12)(34)&(123)&(12345)&(12354)\cr
  \rm size&1&15&20&12&12\cr
\hline
  f_2 &300&32&63&65&65\cr
  P_2 &8.3\% & 0.8\% & 1.75\% & 1.8\% & 1.8\% \cr
  f_3 &1320&24&12&20&20\cr
  P_3 &0.6\% & 0.01\% &0.005\% &0.009\% &0.009\% \cr
  Q_3 &55\% &1\% &0.5\% &0.8\% &0.8\%	 
\end{array}
\]

We will close this section showing that the distribution $Q_3$ over the symmetric group 
$\Sigma_n$ is always postive for even permutations when $n\geq 3$. 
The following Lemma is straightforward and we will omit its proof.

\begin{lemma}\label{commutators_and_products} 
Suppose that $(x_1,x_2,x_3)$ and $(y_1,y_2,y_3)$ are 3-tuples of elements from $G$ that satisfy 
$[x_i,x_j]=g$ and $[y_i,y_j]=h$ for $i<j$.
If $[x_i,y_j]=1$ for all $i$ and $j$, then the 3-tuple $(x_1 y_1,x_2 y_2, x_3 y_3)$ satisfies 
$[x_i y_i, x_j y_j] = gh$ for $i<j$.
\end{lemma}

\noindent {\bf Proof of Theorem~\ref{thm_three}:}
We want to prove that if $n\geq 3$ and $g\in A_n$, then $f_{3,\Sigma_n}(g) > 0$. 
We will proceed by induction on $n$. Using Theorem~\ref{thm_two} one can check that for $\Sigma_3$ and $\Sigma_4$ 
we have the following:
\begin{enumerate}
\item $f_{3,\Sigma_3}(1)=48$,
\item $f_{3,\Sigma_3}((123))=6$,
\item $f_{3,\Sigma_4}((12)(34))=72$, and
\item $f_{3,\Sigma_4}((123))=12$.
\end{enumerate}
Now assume the result is true for integers greater than 4 and less than $n$. 
We write $g$ as a product of disjoint cycles: $g=\sigma_1\cdots \sigma_u \tau_1\cdots\tau_v$, so that each 
$\sigma_i$ is a $r_i$-cycle of even length, and each $\tau_i$ is a $t_i$-cycle of odd length (including 1-cycles). 
Thus $n=r_1+\cdots+ r_u+t_1+\cdots+t_v$,
and $u$ must be an even number equal to 0, or greater than or equal to 4. 
By relabeling if necessary, we can assume that $g\in A_r \times A_t$, 
where $r=r_1+\cdots + r_u$ and $t=t_1+\cdots +t_v$. 

Then we have to consider two cases:
\begin{enumerate}
\item If each of the cycles $\sigma_i$ and $\tau_j$ has length less than $n$, then both $u$ and $v$ are less than $n$ 
and $u$ is equal to 0, or greater than or equal to 4. Then by inductive hypothesis there exist triples 
$(x_1,x_2,x_3)$ in $\Sigma_u$ and $(y_1,y_2,y_3)$ in $\Sigma_v$ such that 
$[x_i,x_j] = \sigma_1\cdots \sigma_u$ and $[y_i,y_j] = \tau_1\cdots\tau_v$ for all $i<j$. 
Hence by Lemma~\ref{commutators_and_products} it follows that $g=[x_iy_i,x_jy_j]$ for all $i<j$, as wanted.

\item If one of the cycles of $g$ has length equal to $n$, then this implies that $g$ must be an 
$n$-cycle. Since $n$-cycles are conjugate in $\Sigma_n$ it suffices to prove that $f_{3,\Sigma_n}((12\cdots n))>0$. 
According to Theorem~\ref{f3rotations} and Proposition~\ref{properties} we have: 
\[ 
f_{3,\Sigma_n}((12\cdots n))\geq f_{3,D_{2n}}((12\cdots n))>0,\] as wanted.
\end{enumerate}
\hfill $\square$

\begin{remark}\label{ore_extending}\normalfont
We have conjectured that $f_{3,A_n}(g)>0$ for all $g$ in $A_n$ when $n\geq 5$. This conjecture is in a way an extension
of Ore's conjecture. More precisely, if we set 
\[
\mathcal{O}_{k,G} = \{ g\in G \colon f_k(g) >0\},
\]
then Ore's conjecture states 
that $\mathcal{O}_{2,G} = G$ for any non-abelian simple group $G$. With this notation, Theorem~\ref{thm_three} states
that $\mathcal{O}_{3,\Sigma_n} = A_n$ for $n\geq 3$. We conjecture that 
$\mathcal{O}_{3,A_n} = A_n$ for $n\geq 5$.

On the other hand we know that $f_{3,A_5}((123))=12$, so from Proposition~\ref{properties} and the fact that 3-cycles
are conjugate in $A_n$ when $n\geq 5$ it follows that 
$f_{3,A_n}(g)>0$ whenever $g$ is the identity or a 3-cycle. Moreover, since any permutation in $A_n$ can be 
written as a product of at most 
$n/2$ 3-cycles, it follows that the support of the $k$-iterated convolution product $Q_3^{*k}$ is equal to $A_n$ 
for $k\geq n/2$. Therefore,
the random walk driven by $Q_3$ is ergodic and $Q_3^{*k}$ converges in the $L_1$-norm to the uniform distribution. 
This latter in turn implies that if we fix both $\epsilon>0$ and $n\geq 5$, then we can find $k$ large enough (see Corollary 1.2 of~\cite{shalev}) so 
that 
\[
|A_n|\geq |\mathcal{O}_{A_n}^{*k}| \geq (1-\epsilon)|A_n|,     
\]
where $\mathcal{O}_{A_n}^{*k} = \{ g\in A_n \colon Q_3^{*k}(g)>0 \}$. This is remarkable to some extent,
because it is telling that almost every element in $A_n$ can be written as a product of $k$ factors
each of which is in $\mathcal{O}_{3,A_n}$, for some $k$ large enough. 

As for $\mathcal{O}_{k,G}$ when $k>3$ we have computational evidence to conjecture that 
$\mathcal{O}_{k,\Sigma_n} = \{ 1 \} $ when $k>3, n>2$. Note that it suffices to show the case when $k=4$ 
(since $f_{n}(g) \neq 0$ implies $f_{n-1}(g) \neq 0$).

\end{remark}

\end{document}